\documentclass[reqno,11pt]{amsart}
\usepackage{amsmath, latexsym, amsfonts, amssymb, amsthm, amscd}

\usepackage{xcolor}

\usepackage[utf8]{inputenc}

\RequirePackage{fix-cm}
\usepackage{amsmath, latexsym, amssymb}
\usepackage{graphicx}
\usepackage{hyperref}
\usepackage{graphics,epsf,psfrag}
\usepackage{caption}
\usepackage{subcaption}
\numberwithin{equation}{section}

\newtheorem{theorem}{Theorem}[section]

\newtheorem{proposition}[theorem]{Proposition}
\newtheorem{cor}[theorem]{Corollary}
\newtheorem{rem}[theorem]{Remark}

\newtheorem{hypothesis}[theorem]{Hypothesis}

\newcommand{\ind}{\mathbf{1}}

\renewcommand{\tilde}{\widetilde}

\newcommand{\cP}{{\ensuremath{\mathcal P}} }

\newcommand{\cL}{{\ensuremath{\mathcal L}} }

\newcommand{\cW}{{\ensuremath{\mathcal W}} }

\newcommand{\bP}{{\ensuremath{\mathbf P}} }
\newcommand{\bE}{{\ensuremath{\mathbf E}} }

\DeclareMathSymbol{\leqslant}{\mathalpha}{AMSa}{"36} 
\DeclareMathSymbol{\geqslant}{\mathalpha}{AMSa}{"3E} 
\DeclareMathSymbol{\eset}{\mathalpha}{AMSb}{"3F}     
\renewcommand{\leq}{\;\leqslant\;}                   
\renewcommand{\geq}{\;\geqslant\;}                   
\newcommand{\maxtwo}[2]{\max_{\substack{#1 \\ #2}}} 

\newcommand{\bbR}{{\ensuremath{\mathbb R}} }

\newcommand{\norm}[1]{\left\lVert#1\right\rVert}

\title{A note on Fokker-Planck equations and graphons}

\email{fabio.coppini@unifi.it}

\author{Fabio Coppini}
\address{Dipartimento di Matematica Ulisse Dini, {Università degli Studi di Firenze}, {Viale Giovanni Battista Morgagni, 67/a}, {Florence}, {50134}, {Italy}}

\begin{document}

\begin{abstract}
{Fokker-Planck equations represent a suitable description of the finite-time behavior for a large class of particle systems as the size of the population tends to infinity. Recently, the theory of graph limits has been introduced in the classical mean-field framework to account for heterogeneous interactions among particles. In many instances, such network heterogeneity is preserved in the limit which turns from being a single Fokker-Planck equation (also known as McKean-Vlasov) to an infinite system of non-linear partial differential equations (PDE) coupled by means of a graphon. While appealing from an applied viewpoint, few rigorous results exist on the graphon particle system. This note addresses such limit system focusing on the relation between interaction network and initial conditions: if the system initial datum and the graphon degrees satisfy a suitable condition, a significantly simpler representation of the solution is available.
	
	Interesting conclusions can be drawn from this result: substantially different graphons can give rise to exactly the same particle behavior, e.g., if the initial condition is independent of the graphon, the particle behavior is indistinguishable from the well-known mean-field limit as long as the graphon has constant degree; step-kernels, a building block of graph limit theory, can be used to approximate the graphon particle system with a finite system of Fokker-Planck equations and an explicit rate of convergence.
}
\end{abstract}

\maketitle

\section{Introduction, model and literature}

\subsection{Introduction, aim of this note and organization}
In the last years, the study of interacting particle systems with a non-trivial dense network structure has been repeatedly addressed in the mathematical community: see, e.g., \cite{cf:BBW,cf:DGL,cf:RO,cf:CDG} for interacting diffusions,  \cite{cf:CH18,cf:CCGL19,cf:parise20} for applications in mean-field games and \cite{cf:KHT20,cf:delmas20} in the context of dynamical systems. Depending on the setting, many results on interacting particle systems are nowadays available \cite{cf:BCW20,cf:BCN20,cf:C19,cf:CM19,cf:L18} whenever the underlying graph sequence is converging, in some sense depending case by case, to a suitable object. More precisely, if the graph limit is a graphon (e.g., \cite{cf:LS06,cf:Lov}), then, as the size of the system tends to infinity, the finite-time population behavior is suitably described by an infinite system of coupled non-linear Fokker-Planck equations, the coupling between equations being made by the graphon limit itself (see equation \eqref{eq:fp} for a first example).

This note addresses the graphon particle system obtained in the limit and tries to better understand the relationship between the initial conditions and the choice of the underlying network. Such a perspective is important not only to better understand the limit system, but also in case one is interested in reconstructing the network structure by looking at the particle dynamics: we show that different graphons can lead to exactly the same particle behavior.

Graphons have been used as a model for many real-world networks, yet, to the author's knowledge, known results on graphon particle systems are mainly limited to existence and uniqueness of solutions; we refer to Sect.~ \ref{ss:literature} for the current literature.

\medskip

A labeled graphon (we use the same notation of \cite{cf:Lov}) is a symmetric measurable function $W$ defined on the unit-square
\begin{equation*}
\begin{split}
W: [0,1]^2 &\to [0,1]\\
(x,y) &\mapsto W(x,y).
\end{split}
\end{equation*}
If the limit population is represented as a continuum of particles labeled by the unit interval $[0,1]$, then $W(x,y)$ stands for the connection strength between the particle labeled with $x$ and the one labeled with $y$. The function $W$ thus describes the connection network underlying an infinite population of particles.

Fix a finite time horizon $T>0$. For one-dimensional particles which are interacting through a regular function $\Gamma$, the graphon particle system is an infinite system of partial differential equations (PDE) coupled by means of $W$, i.e., 
\begin{equation}
\label{eq:fp}
\begin{split}
\partial_t \mu_t (\theta, x) & = \frac 12 \partial^2_\theta \mu_t (\theta, x) \\
-& \partial_\theta \left[ \int_0^1 W(x,y) \mu_t(\theta, x) \int_\bbR \Gamma(\theta, \theta') \mu_t(d \theta', y) d y \right], \quad x \in [0,1],
\end{split}
\end{equation}
for $t \in [0,T]$ and where $\mu = \{ \mu_t(\cdot, x), t \in [0,T]\}_{x \in [0,1]}$ is a collection of probability measures. The initial datum is given by a probability measure $\mu_0 \in \cP(\bbR \times [0,1])$.  From the disintegration theorem (\cite[\S 5]{dudley}), we have that $\mu_0(\cdot, x) \in \cP(\bbR)$ for almost every $x \in [0,1]$. Namely, the probability measure 
\begin{equation*}
\mu(x) := \mu(\cdot, x) = \{\mu_t(\cdot, x), t \in [0,T] \} \in \cP(C([0,T], \bbR))
\end{equation*}
represents the law of the trajectory associated to the $x$-labeled particle, this last one being connected to the others by means of $W(x, \cdot) : [0,1] \to [0,1]$, see the following \eqref{eq:family}. Each trajectory is a continuous function from $[0,T]$ to $\bbR$, i.e., an element of the space $C([0,T],\bbR)$.

To the system of non-linear Fokker-Planck equations \eqref{eq:fp}, it is associated a family of continuous processes $\{\theta^x\}_{x \in [0,1]} \subset C([0,T], \bbR)$, which formally solve
\begin{equation}
\begin{cases}
\label{eq:family}
\theta^x_t = \theta^x_0 + \int_0^t \int_0^1 W(x,y) \int_\bbR \Gamma(\theta^x_s, \theta) \mu_s(d \theta, y) d y \, d s + B^x_t, \quad t \in [0,T], \\
\mu_t(\cdot, y) = \cL(\theta^y_t), \quad \text{ for } t \in [0,T] \text{ and } y \in [0,1],
\end{cases}
\end{equation}
where the law of the initial condition $\theta^x_0$, denoted by $\cL(\theta^x_0)$, is thus given by $\mu_0(x)$. The family $\{ B^x\}_{x\in [0,1]}$ is composed of independent and identically distributed (IID) Brownian motions, independent of the initial conditions as well.

The link between \eqref{eq:fp} and \eqref{eq:family} is given by the fact that $\mu_t(x) = \cL(\theta^x_t)$ for every $t \in [0,T]$ and $x \in [0,1]$. We refer to Proposition \ref{p:existenceUniqueness} for a precise statement which includes more general systems, including the degenerate case ($\sigma \equiv 0$).

\medskip

System \eqref{eq:fp} or, equivalently, the family of non-linear processes \eqref{eq:family}, have been proposed as a limit description in the literature \cite{cf:L18,cf:RO,cf:CH18,cf:BCW20,cf:BCN20}, yet very little is known on their mathematical properties. We thus aim at mathematically addressing \eqref{eq:fp}, firstly by showing the strong link with the graphon theory (especially including unlabeled graphons, see Proposition \ref{p:unlabeled}), secondly by proving that a simpler representation for both \eqref{eq:fp} and \eqref{eq:family} exists, whenever the initial conditions and the underlying graphon satisfy a suitable condition, see Theorem \ref{thm} and, in particular, Corollary \ref{cor:representation} and Proposition \ref{p:stepgraphon}.

\medskip

The note is organized as follows: the next subsection presents the general model associated to \eqref{eq:family}, together with known results from the existing literature. Related works are discussed at the end of this first section.

Sect.~\ref{s:thm} provides the main result, Theorem \ref{thm}, as well as two corollaries and a relevant application to step kernels, see Proposition \ref{p:stepgraphon}. Notably,  Corollary \ref{cor:mean_field} addresses the classical mean-field scenario showing that, if the graphon has constant degree, i.e., such that $d(\cdot) = \int_0^1 W(\cdot,y)d y$ is constant, then the dynamics is mean-field.

Sect.~\ref{s:proof} contains the mathematical set-up and the proof of the main result.

\subsection{The model and some known results}
We consider a class of models slightly more general than \eqref{eq:family}: fix $\mu_0 \in \cP(\bbR \times [0,1])$ and let $\{\theta^x\}_{x \in [0,1]}$ be the family solving the $\infty$-dimensional coupled system:
\begin{equation}
\begin{cases}
\label{eq:theta}
\theta^x_t = \theta^x_0 + \int_0^t F(\theta^x_s) d s + \int_0^t \int_0^1 W(x,y) \int_\bbR \Gamma(\theta^x_s, \theta) \mu_s(d \theta, y) d y \, d s + \int_0^t \sigma(\theta^x_s) d B^x_s \\
\mu_t(\cdot, y) = \cL(\theta^y_t), \quad \text{ for } t \in [0,T] \text{ and } y \in [0,1],
\end{cases}
\end{equation}
where $F$, $\Gamma$ and $\sigma$ are 1-Lipschitz functions bounded by 1, and $\{ B^x \}_{x \in [0,1]}$ is a family of IID Brownian motions on $\bbR$.  For every $x \in [0,1]$, the initial condition $\theta^x_0$ is a random variable with law given by $\cL(\theta^x_0) = \mu_0(x)$, independent of the other initial conditions and of the Brownian motions. The probability measure induced by the initial conditions and the family of Brownian motions is denoted by $\bP$ and the corresponding expectation by $\bE$.

\medskip

We work under the following assumptions.

\begin{hypothesis}
	\label{hyp}
	We assume that  $\bE \left[ \vert \theta^x_0(0) \vert^2 \right] < \infty$
	\begin{enumerate}
		\item (measurability) The map $[0,1] \ni x \mapsto \mu_0(x) \in \cP(\bbR)$ is measurable;
		\item (moment condition) $\bE [ \vert  \theta^x_0 (0) \vert ^2] < \infty$ for every $x \in [0,1]$.
	\end{enumerate}
\end{hypothesis}

\begin{rem}
	\label{rem:hypothesis}
	Measurability with respect to $x$ is necessary in order to work with graphons and to integrate the marginal laws $\mu(x)$ with respect to $x\in[0,1]$.
	
	The finite second moment condition is related to the use of the 2-Wasserstein distance between probability measures (see the Sect.~\ref{s:proof} and, in particular, \eqref{d:wass}). The choice $p=2$ is purely arbitrary, it is obviously possible to work with initial conditions with finite $p$-moment for every $p\geq 1$.
\end{rem}

We stick to the one-dimensional setting, i.e., $\theta^x$ taking values in $\bbR$; however, all the proofs presented below are easily extendable to any finite dimension.

\medskip

Existence and uniqueness for system \eqref{eq:theta} are known.

\begin{proposition}
	\label{p:existenceUniqueness}
	Under the measurability assumption and the moment condition in Hypothesis \ref{hyp}, there exists a unique pathwise solution to system \eqref{eq:theta}.
	
	Moreover, if we denote by $\mu(x)$ the law of $\theta^x$ for each $x \in [0,1]$, then the map $[0,1] \ni x \mapsto \mu(x) \in \cP(C([0,T], \bbR))$ is measurable and $\mu(x)$ weakly solves 
	\begin{equation}
	\label{eq:fpBIG}
	\begin{split}
	\partial_t \mu_t (\theta, x) \, &=\,  \frac  12 \partial^2_{\theta }\left({\sigma^2(\theta)} \mu_t (\theta, x)\right) - \partial_{\theta} \left( \mu_t (\theta, x) F(\theta) \right) \\
	-& \partial_\theta \left[ \int_0^1 W(x,y) \mu_t (\theta, x) \, \int_{\bbR} \Gamma(\theta, \theta') \mu_t (d \theta', y) d y \right] ,
	\end{split}
	\end{equation}	
	for every $x \in [0,1]$.
\end{proposition}
We refer to \cite[Proposition 2.1]{cf:BCW20} and \cite[Proposition 2.4]{cf:L18} for two standard similar proofs.  We point out that $W$ needs only to be bounded and not necessarily with values in $[0,1]$, i.e, existence and uniqueness hold for every kernel $W$, as in the notation below.

\medskip

Let $\cW := \{W:[0,1]^2 \to \bbR \text{ bounded, symmetric and measurable}\}$ be the space of kernels\footnote{we always consider two kernels to be equal if and only if they differ on a subset of Lebesgue measure zero. We follow closely the notation in \cite{cf:Lov}.}. The cut-norm of $W\in\cW$ is defined as
\begin{equation}
\label{d:cut-norm}
\norm{W}_{\square} := \max_{S,T \subset [0,1]} \left\vert  \int_{S\times T} W(x,y) d x d y \right\vert 
\end{equation}
where the maximum is taken over all measurable subsets $S$ and $T$ of $I$.  Let $\cW_0 := \{W \in \cW : 0 \leq W \leq 1 \}$ be the space of labeled graphons: for $W,\, V \in \cW_0$ their cut-distance is defined by
\begin{equation}
\label{d:cut-distance}
\delta_\square (W,V) := \, \min_{\varphi \in S_{[0,1]}} \, \norm{W - V^\varphi}_\square,
\end{equation}
where the minimum ranges over $S_{[0,1]}$ the space of invertible measure preserving maps from $[0,1]$ into itself and where $V^\varphi (x,y) := V(\varphi(x),\varphi(y))$ for $x,y \in [0,1]$.

The cut-distance $\delta_\square$ is a pseudometric on $\cW_0$ since it can be zero between two different labeled graphons. If we identify all labeled graphons with cut-distance zero, we obtain the space of graphons $\tilde{\cW}_0 := \cW_0 / \delta_\square$. A well-known result of graph limits theory says that $(\tilde{\cW}_0, \delta_\square)$ is a compact metric space \cite[Theorem 9.23]{cf:Lov}.

Proposition \ref{p:existenceUniqueness} can thus be restated by saying that for every kernel $W \in \cW$, there exists a unique solution $\mu^W$ to \eqref{eq:fpBIG}. It is thus natural to ask whether the application $W \mapsto \mu^W$ is continuous (e.g., in the topology of the weak-convergence) with respect to $\norm{\cdot}_\square$. This point is discussed in the next remark.

\begin{rem}
	\label{rem:continuity}
	Under suitable assumptions on the coefficients, see \cite[Proposition 3.3]{cf:BCN20} but also \cite[Theorem 2.1]{cf:BCW20} and \cite[Lemma 2.7]{cf:KM17}, it is possible to show that for two solutions $\mu^W$ and $\mu^V$ associated to $W$ and $V$ in $\cW$ respectively, it holds that
	\begin{equation}
	\label{eq:continuity}
	D^2_T (\mu^W, \mu^V) \leq C \norm{ W-V}_\square, \quad \text{ for some } C>0,
	\end{equation}
	where $D_T$ is some distance on probability measures metricizing the weak-convergence. Equation \eqref{eq:continuity} establishes that the application $W \mapsto \mu^W$ is (Hölder-)continuous and, in particular, that similar graphons \textendash{} in the cut-norm \textendash{} lead to similar particle behaviors.
	
	Observe that it is not possible to infer that different graphons (possibly in $\delta_{\square}$-distance) lead to different behaviors. As Theorem \ref{thm} and its corollaries show, this cannot be true. There are relevant examples where it is possible to prove that on substantially different graphons, the particle dynamics can coincide.
\end{rem}

For every $W \in \cW_0$, Proposition \ref{p:existenceUniqueness} shows that the infinite system \eqref{eq:fpBIG} admits a unique solution, yet it does not address the unlabeled class of $W$ in $\tilde{\cW}_0$. We explicit such relation in the next proposition.

\begin{proposition}
	\label{p:unlabeled}
	Let $U$ be a uniform random variable on $[0,1]$ and $\theta^U$ the non-linear process in $\{ \theta^x \}_{x \in [0,1]}$ with random label $U$. Then, the law of $\theta^U$ is given by $\tilde{\mu} = \int_0^1 \mu(x) d x$ and it is independent of the class of $W$ in $\tilde{\cW}_0$.
\end{proposition}

\begin{proof}
	This result was firstly stated in \cite[Proposition 2.1]{cf:BCN20} in the case of interacting oscillators. In case the particles are living in $\bbR$ the proof does not change. Indeed, let $\varphi$ be a measure preserving map from $[0,1]$ to itself, we observe that $\mu(\varphi(x))$ solves:
	\begin{equation}
	\label{eq:fpTILDE}
	\begin{split}
	\partial_t \mu_t (\theta, \varphi(x)) \, &=\,  \frac  12 \partial^2_{\theta }\left({\sigma^2(\theta)} \mu_t (\theta, \varphi(x))\right) - \partial_{\theta} \left( \mu_t (\theta, \varphi(x)) F(\theta) \right) \\
	-& \partial_\theta \left[ \int_0^1 W(\varphi(x),y) \mu_t (\theta, \varphi(x)) \, \int_{\bbR} \Gamma(\theta, \theta') \mu_t (d \theta', y) d y \right] ,
	\end{split}
	\end{equation}
	where the last term is equal to
	\begin{equation*}
	\partial_\theta \left[ \int_0^1 V(x,y) \mu_t (\theta, \varphi(x)) \, \int_{\bbR} \Gamma(\theta, \theta') \mu_t (d \theta', \varphi(y)) d y \right]
	\end{equation*}
	with $V(x,y) = W(\varphi(x), \varphi(y))$. Thus $\{\mu(\varphi(x)) \}_{x \in [0,1]}$ solves the same system of $\{\nu(x)\}_{x \in [0,1]}$ solution to \eqref{eq:fpBIG} with $V \in \cW_0$ and initial condition $\{\mu_0(\varphi(x))\}_{x \in [0,1]}$. Clearly $\tilde{\mu} = \int_0^1 \mu(x) d x = \int_0^1 \mu(\varphi(x)) d x = \int_0^1 \nu(x) d x =  \tilde{\nu}$.
\end{proof}

Proposition \ref{p:unlabeled} makes clear that, as for a graphon the relevant information does not depend on the chosen labeling, the same holds true for the marginal law of the graphon particle system: it does not change under relabeling of the particles. With the exception of \cite{cf:BCN20} for interacting oscillators, this result has never been stated in the literature.

\begin{rem}
	\label{rem:initialConditions}
	Observe that two different initial conditions $\mu^1_0, \mu^2_0 \in \cP(\bbR \times [0,1])$ but such that $ \int_0^1 \mu^1_0 (x) d x = \int_0^1 \mu^2_0(x) d x$, lead to two different solutions $\mu^1$ and $\mu^2$ to equation \eqref{eq:fpBIG}. Following the notation of Proposition \ref{p:unlabeled}, the corresponding marginals $\tilde{\mu}^i = \int_0^1 \mu^i(x) d x$, for $i=1,2$, are in general not equal.
	
	In particular, by choosing $\mu^2_0 \equiv  \int_0^1 \mu^1_0 (x) d x$, one can derive the following statement: the behavior of a uniform random particle can be described by two different solutions, i.e., $\tilde{\mu}^2$ in case only the marginal law of the initial conditions is known, or $\tilde{\mu}^1$ in case the joint law is known.
	
	In conclusion, one cannot interchange expectation related to $U$ and the dynamics semigroup. 	The way in which the dynamics mixes the initial conditions is an interesting challenge for further studies.
\end{rem}

\subsection{Related works}
\label{ss:literature}
The graphon framework \cite{cf:LS06,cf:Lov} has been introduced in the theory of particle systems \cite{cf:CM19,cf:LS14,cf:medvedev,cf:RO} as an useful ingredient to construct inhomogeneous random graph sequences with nice statistical properties (edge independence, graph homomorphism, etc.).  Most of the known literature on particle systems has been focused on the converging properties of particle systems on such random graph sequences, see, e.g., \cite{cf:BCW20,cf:BCN20,cf:L18,cf:RO}. They show that some of the classical mean-field arguments (as propagation of chaos \cite{cf:szni}) can be extended to deal with random graphs and to include labeled graphons in the limit description. Nevertheless, only little attention has been put into the study of the limit object \eqref{eq:fpBIG} which remains, to the author's knowledge, rather unknown.

The recent work \cite{cf:BCN20} shows a direct connection between particle systems and the graphon theory: under suitable hypothesis on \eqref{eq:family}, there exists a Hölder-continuous mapping between the space of graphons $(\tilde{\cW}_0, \delta_\square)$ and $\tilde{\mu}$ as in Proposition \ref{p:unlabeled}. This allows to consider general graph sequences, as exchangeable random graphs \cite{cf:DJ08}, which leads to a (possibly) random graphon $W$ in the limit. However, despite the rather understood convergence properties, only a few insights are available on the limit particle system, we refer to the examples in Sect.~2.3 of \cite{cf:BCN20}.

Although interesting by itself, the study of the limit Fokker-Planck equation does not necessarily provide a suitable description of particle systems on diverging time scales, as already raised in \cite{cf:CDG,cf:C19} for the mean-field case. With the exception of dissipative dynamics \cite{cf:BW20}, a substantial understanding of the phase space of \eqref{eq:fpBIG} is needed to study the long-time behavior of finite particle systems on graphs as shown in \cite{cf:C19}. It is thus important to better understand system \eqref{eq:fpBIG} in order to address the finite particle system counterpart.

%
%

\begin{figure}[h]
	\centering
	\begin{subfigure}[b]{0.32\textwidth}
		\centering \includegraphics[scale=0.65]{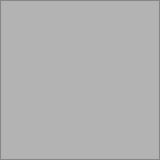}
		\caption{\label{fig:mean_field}Constant graphon}
	\end{subfigure}
	\begin{subfigure}[b]{0.32\textwidth}
		\centering \includegraphics[scale=0.65]{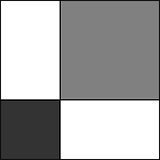}
		\caption{Disconnected graphon}\label{fig:disconnected}
	\end{subfigure}
	\begin{subfigure}[b]{0.32\textwidth}
		\centering \includegraphics[scale=0.65]{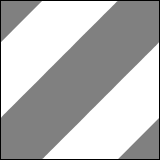}
		\caption{Cayley graphon}\label{fig:cayley}
	\end{subfigure}
	\caption{\label{fig:graphons}Examples of graphons that satisfy \eqref{Hmeanfield} (constant degree). Observe that the graphon in (B) is composed of two connected components, the smaller one being more densely connected (darker) with respect to the other one so to satisfy \eqref{Hmeanfield}.
		Finally, observe that (A), (B) and (C) are different both as labeled and unlabeled graphons, i.e., different in $\delta_\square$-distance.}
\end{figure}

\section{Main result, examples and applications}
\label{s:thm}
%

\subsection{Main result}
Let the initial condition $\mu_0 \in \cP(\bbR \times [0,1])$ and the labeled graphon $W \in \cW_0$ be fixed. Before stating the main result, we give two definitions.

For $x \in [0,1]$, let  $d(x) = \int_0^1 W(x,y) d y$ be the degree of $x$ in $W$. More generally, define the degree of $x$ with respect to a (measurable) subset $A\subset[0,1]$ by
\begin{equation}
d_A(x) = \int_A W(x,y) d y.
\end{equation}

Let $\mathcal{J}$ be a partition of $[0,1]$ indexed by $J$. For $\bar{z} \in [0,1]$, we denote by $[\bar{z}]$ the corresponding cell in $\mathcal{J}$, i.e., the unique $A \in \mathcal{J}$ such that $\bar{z} \in A$. We say that $\mathcal{J} = \{A_{\bar{z}} \}_{\bar{z} \in J}$ is \textit{admissible} if every $A \in \mathcal{J}$ is a measurable set of $[0,1]$. Without loss of generality, we suppose that $J \subset [0,1]$ and that for every $\bar{z} \in J$, $\bar{z} \in A_{\bar{z}}$, i.e., $[\bar{z}] = A_{\bar{z}}$.

\begin{theorem}
	\label{thm}
	Assume Hypothesis \ref{hyp}. Suppose that there exists an admissible partition $\mathcal{J} = \{[\bar{z}]\}_{\bar{z} \in J}$ of $[0,1]$ such that, for every $\bar{x}$ and $\bar{y}$ in $J$ it holds that
	\begin{equation}
	\label{H}
	d_{[\bar{y}]}(x) = \int_{[\bar{y}]} W(x,y) d y = \int_{[\bar{y}]} W(\bar{x}, y) d y = d_{[\bar{y}]}(\bar{x}) \quad \text{ for all } x \in [\bar{x}].
	\end{equation}
	Further assume that $\mu_0$ is constant on every $[\bar{z}] \in \mathcal{J}$.
	Then, $\mu$ solution to \eqref{eq:fpBIG} is constant on every $[\bar{z}] \in \mathcal{J}$, i.e., $\mu(x) = \mu(\bar{x})$, for all $x \in [\bar{x}]$ and every $\bar{x} \in J$.
\end{theorem}

Condition \eqref{H} is requiring that, if we partition the interval $[0,1]$ in $J$ subsets $\{[\bar{x}]\}_{\bar{x} \in J}$, then the density of neighbors with respect to any of these subsets, i.e., $d_{[\bar{y}]} (\cdot)$ for some $\bar{y} \in J$, is piecewise constant on $[0,1]$ and completely characterized by the values on $J$. See Figure \ref{fig:graphons} (B) and  Figure \ref{fig:stepgraphon} (A) for graphons that satisfy \eqref{H}.

Theorem \ref{thm} basically states that, if one can group the particles so that within each group they have the same initial condition (in law) and are \textit{equi-connected} to the other groups in the sense of \eqref{H}, then the solution $\mu$ to \eqref{eq:fpBIG} is constant on these groups. Observe that Theorem \ref{thm} does not depend on the representative of the class of $W \in \tilde{\cW_0}$. As a consequence, one can relabel the particles such that $\mu$ is piecewise constant as a function of $x \in [0,1]$.


\medskip

A direct consequence of Theorem \ref{thm} is that the solution to \eqref{eq:fpBIG} can be obtained solving a system of coupled differential equations indexed by $J$ instead of $[0,1]$. When $|J|< \infty$, this implies that \eqref{eq:fpBIG} is equivalent to a finite  system of differential equations.

\begin{cor}
	\label{cor:representation}
	Under the assumptions of Theorem \ref{thm}. The infinite system \eqref{eq:fpBIG} is suitably described by the system of coupled partial differential equations
	\begin{equation}
	\label{eq:fpREDUCED}
	\begin{split}
	\partial_t \nu_t (\theta, \bar{x}) \, &=\,  \frac  12 \partial^2_{\theta }\left({\sigma^2(\theta)} \nu_t (\theta, \bar{x})\right) - \partial_{\theta} \left( \nu_t (\theta, \bar{x}) F(\theta) \right) \\
	-& \partial_\theta \left[ \int_J W(\bar{x},\bar{y}) \nu_t (\theta, \bar{x}) \, \int_{\bbR} \Gamma(\theta, \theta') \nu_t (d \theta', \bar{y}) d \bar{y} \right], \quad \bar{x}\in J,
	\end{split}
	\end{equation}
	with $\nu_0 = \mu_0$. Notably, $\mu(x) = \nu(\bar{x})$ for every $\bar{x} \in J$ and $x \in [\bar{x}]$.
\end{cor}

\begin{proof}
	Existence and uniqueness for system \eqref{eq:fpREDUCED} directly follow from Proposition \ref{p:existenceUniqueness}. Moreover, system \eqref{eq:fpREDUCED} is written in closed form and can be thus solved independently of \eqref{eq:fpBIG}.
\end{proof}

Graphons have proven to be an important tool for establishing the convergence of particle systems on general graph sequences. Notably, it is possible to show that interacting particles on apriori different graph sequences have the same asymptotic behavior whenever the graph limits coincide in $\tilde{\cW_0}$, the infinite particle system \eqref{eq:theta} being indistinguishable.

We want to stress that the representation \eqref{eq:fpBIG} can be slightly misleading when comparing two systems: Theorem \ref{thm} and Corollary \ref{cor:representation} show that different unlabeled graphons satisfying \eqref{H}, lead to the same solution no matter the formal difference in the representation  \eqref{eq:fpBIG}. In other words, different graph sequences leading to different graph limits do not imply different particle behaviors. Understanding which are the key properties of graphons that fully characterize a complex system behavior is an interesting subject for future research.

\subsection{The case of constant degree and uniform initial conditions}

We turn to an interesting application when the graphon $W$ has constant degree and the law of the initial condition is constant with respect to $x\in [0,1]$. As already expressed in Remark \ref{rem:initialConditions}, assuming constant initial conditions covers the case of a uniform random particle in \eqref{eq:fpBIG} when only the marginal law of the initial conditions, i.e., $\int_0^1 \mu_0(x) d x$, is available.

\begin{cor}
	\label{cor:mean_field}
	Assume Hypothesis \ref{hyp}. Suppose that there exists $p \in [0,1]$ such that
	\begin{equation}
	\label{Hmeanfield}
	p = \int_0^1 W(x,y) d y, \quad \text{ for all } x \in [0,1],
	\end{equation}
	and that the initial is a constant function of $x$, i.e., $\mu_0(x) \equiv \mu_0$.
	
	Then, $\{\mu(x)\}_{x \in [0,1]}$ is label independent, i.e., $\mu(\cdot) \equiv \mu \in \cP(C([0,T], \bbR))$, and $\mu$ solves the classical McKean-Vlasov equation
	\begin{equation}
	\label{eq:mcv}
	\partial_t \nu_t (\theta) \, =\,  \frac  12 \partial^2_{\theta }\left({\sigma^2(\theta)} \nu_t (\theta)\right) - \partial_{\theta} \left( \nu_t (\theta) F(\theta) \right) -  \partial_\theta \left[  \nu_t (\theta) \, p \int_{\bbR} \Gamma(\theta, \theta') \nu_t (d \theta') \right],
	\end{equation}
	with initial condition $\mu_0$.
\end{cor}

\begin{proof}
	The hypothesis on $W$ allows to consider an admissible partition made of a single set, i.e., $|J|=1$. Corollary \ref{cor:representation} yields the result.
\end{proof}

The constant degree assumption on $W$ is satisfied by a large class of non-trivial graphons as the examples shown in Figure \ref{fig:graphons}.

\begin{rem}
	\label{rem:DGL}
	Observe that Corollary \ref{cor:mean_field} can be derived from the results in \cite{cf:DGL} combined with the ones on particle systems on graphons, e.g., \cite{cf:L18,cf:BCW20}. Indeed, in \cite{cf:DGL} it is shown that, for a particle system defined on a graph sequence and with IID initial conditions, a sufficient condition to obtain the mean-field limit \eqref{eq:mcv} is that each vertex in the (renormalized) graph sequence has the same asymptotic degree density. If one chooses such sequence to converge to a graphon with constant degree, then it satisfies both the hypothesis in \cite[Theorem 1.1]{cf:DGL} and in, e.g., \cite[Theorem 4.1]{cf:BCW20}. Combining \cite[Theorem 1.1]{cf:DGL} and \cite[Theorem 4.1]{cf:BCW20}, we obtain that the same finite particle system converges to \eqref{eq:mcv} but also to \eqref{eq:fpBIG}. As a consequence, the solutions of \eqref{eq:fpBIG} and \eqref{eq:mcv} must be the same, giving an undirect proof of Corollary \ref{cor:mean_field}.
\end{rem}

We show a representative example whenever the graphon is a step-wise constant function, as the one in Figure \ref{fig:stepgraphon}.

\begin{figure}
	\centering
	\begin{subfigure}[b]{0.49\textwidth}
		\centering \includegraphics[scale=0.5]{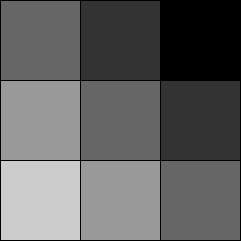}
		\caption{\label{fig:stepgraphon_scalefree}Step kernel}
	\end{subfigure}
	\begin{subfigure}[b]{0.49\textwidth}
		\centering \includegraphics[scale=0.51]{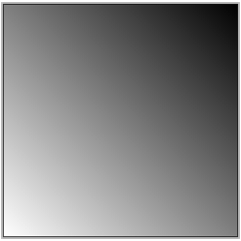}
		\caption{\label{fig:scalefree}Scale free graphon}
	\end{subfigure}
	\caption{\label{fig:stepgraphon}Suppose that $J=\{x_0, x_1, x_2\}$ and that $[x_0] =[0,1/3)$, $[x_1]=[1/3,2/3)$ and $[x_2]=[2/3, 1]$. Then, the step kernel in the figure (A) satisfies \eqref{H}. It also represents a way to approximate the scale-free graphon (B). The grayscale stands for different values in $[0,1]$.}
\end{figure}

\subsection{Finite approximation through step-kernels}
Step kernels (\cite[\S 7.1]{cf:Lov}) represent one of the fundamental blocks of graphon theory since they provide a way to approximate any $W \in \cW_0$ with step-wise constant functions. A function $W \in \cW$ is a step kernel if there is a partition $\cP = \{S_i\}_{i=1, \dots, k}$ of $[0,1]$ into measurable sets such that $W$ is constant on every product set $S_i \times S_j$. We use the following notation
\begin{equation}
\label{d:stepgraphon}
W_{\cP} (x,y) = \sum_{i,j=1}^k w_{ij} \, \ind_{S_i \times S_j} (x,y), \quad \text{ for } x,y \in [0,1],
\end{equation}
where $\{w_{ij} \}_{i,j =1, \dots, k}$ are bounded real numbers. See Figure \ref{fig:stepgraphon} for an example with equidistant partition. 

\medskip


For a step kernel, condition \eqref{H} is clearly satisfied by taking $\mathcal{J} = \cP$, i.e., $J = \{x_i\}_{i=1,\dots,k}$ with $x_i \in S_i$ for every $i=1,\dots,k$.

Applying Theorem \ref{thm} and Corollary \ref{cor:representation}, we have the following representation for the graphon particle system \eqref{eq:fpBIG} on the step graphon $W_\cP$.

\begin{proposition}
	\label{p:stepgraphon}
	Let $W$ be the step kernel in \eqref{d:stepgraphon} and assume that $\mu_0$ is constant on every $S_i \in \cP$. Then the graphon particle system \eqref{eq:fpBIG} is suitably described by the finite collection of probability measures $\{\mu^i\}_{i=1,\dots, k}$ which solve
	\begin{equation}
	\label{eq:SBM}
	\begin{split}
	\partial_t \mu^i_t (\theta) \, =\,  \frac  12 \partial^2_{\theta }\left({\sigma^2(\theta)} \mu^i_t (\theta)\right) - \partial_{\theta} \left( \mu^i_t (\theta) F(\theta) \right) -\\
	\partial_\theta \left[  \sum_{j=1}^k w_{ij} \mu^i_t (\theta) \, \int_{\bbR} \Gamma(\theta, \theta') \mu^j_t (d \theta') \right],
	\end{split}
	\end{equation}
	and where $\mu^i_0 = \mu_0$, for $i=1,\dots,k$.
\end{proposition}

As shown in Proposition \ref{p:stepgraphon}, the representation given in Corollary \ref{cor:representation} becomes natural in the case of step kernels. While Proposition \ref{p:stepgraphon} could be derived with direct computations\footnote{A step kernel can be seen as a Stochastic Block Model, for which the representation \eqref{eq:SBM} is somehow known.}, Theorem \ref{thm} goes a step further: if one replaces every single constant block $S_i \times S_j$ in \eqref{d:stepgraphon} by a (suitable scaled) graphon with constant degree, equation \eqref{eq:SBM} does not change.

We also observe that combining Proposition \ref{p:stepgraphon} with the continuity estimates as in Remark \ref{rem:continuity}, allows to approximate the graphon particle system \eqref{eq:fpBIG} with a finite number of coupled Fokker-Planck equations. Since it is beyond the scope of this note, we do not purse such analysis, yet we make this point slightly more precise in the next remark.

\begin{rem}
	By using the Weak Regularity Lemma {\cite[Corollary 9.13]{cf:Lov}}, one can approximate every graphon by a step kernel with an explicit control on the error. Assuming the continuity of \eqref{eq:fpBIG} with respect to $W$, recall Remark \ref{rem:continuity}, Proposition \ref{p:stepgraphon} opens the way to approximate the infinite graphon particle system \eqref{eq:fpBIG} by using a finite (thus apriori numerically solvable) system of coupled Fokker-Planck equations, with precise bounds on the error.
	
	Observe that the decay rate strongly depends on how well the graphon can be approximated by step kernels. For a general graphon, there exists a step kernel with $k^2$ values, whose cut-distance is at most $2/\sqrt{\log(k)}$, see {\cite[Corollary 9.13]{cf:Lov}}.
\end{rem}

\section{Proof of the main result}
\label{s:proof}

\subsection{Distance between probability measures}
For two probability measures $\bar{\mu},\, \bar{\nu} \in \cP(C([0,T], \bbR))$, define their 2-Wasserstein distance as
\begin{equation}
\label{d:wass}
D_T (\bar{\mu},\bar{\nu})  = \inf_{X,Y} \left\{ \bE \left[ \sup_{t \in [0,T]} \left\vert X_t - Y_t\right\vert^2 \right] : \cL(X) = \bar{\mu}, \, \cL(Y) = \bar{\nu} \right\}^{1/2}
\end{equation}
where the infimum is taken on all random variables $X$ and $Y$ with values in $C([0,T], \bbR)$ and law $\cL$ equal to $\bar{\mu}$ and $\bar{\nu}$ respectively. From \eqref{d:wass} we obtain that for every $s \in[0,T]$ and for every bounded 1-Lipschitz function $f$
\begin{equation}
\label{d:lipsc}
\left\vert\int_\bbR f(\theta) \, \bar{\mu}_s (d \theta) - \int_\bbR f(\theta) \, \bar{\nu}_s (d \theta)\right\vert = \left\vert\int_\bbR f(\theta) \, \left[\bar{\mu}_s (d \theta) -  \bar{\nu}_s (d \theta)\right]\right\vert \leq D_s(\bar{\mu},\bar{\nu}).
\end{equation}
Observe that \eqref{d:wass} also makes sense with $T=0$ and $C([0,T], \bbR)$ replaced by $\bbR$.

\subsection{Proof of Theorem \ref{thm}}
Recall the definition of the 2-Wasserstein distance $D_T$ in \eqref{d:wass}, we aim at showing that
\begin{equation}
\label{goal}
\maxtwo{\bar{x} \in J}{x \in [\bar{x}]} D_T (\mu(\bar{x}), \mu(x)) = 0.
\end{equation}
We first assume $F \equiv 0$ and $\sigma \equiv 1$.

Let $\bar{x} \in J$ and $x \in [\bar{x}]$, without loss of generality, we can suppose that the associated realizations of the Brownian motion are the same. Using the fact that $x \in [\bar{x}]$, the initial conditions cancel and we have that
\begin{equation}
\begin{split}
\theta^x_t - \theta^{\bar{x}}_t =& \int_0^t \int_0^1 W(x,y) \int_\bbR \Gamma(\theta^x_s, \theta) \mu(d \theta, y) d y \, d s \\
&- \int_0^t \int_0^1 W(\bar{x},y) \int_\bbR \Gamma(\theta^{\bar{x}}_s, \theta) \mu(d \theta, y) d y \, d s,
\end{split}
\end{equation}
In particular, this can be rewritten as
\begin{equation}
\begin{split}
& \int_0^t \int_0^1 W(x,y) \int_\bbR \left[ \Gamma(\theta^x_s, \theta) - \Gamma(\theta^{\bar{x}}_s, \theta)\right] \mu(d \theta, y) d y \, d s \\
&+ \int_0^t \int_0^1 \left[W(x,y) - W(\bar{x},y)\right] \int_\bbR \Gamma(\theta^{\bar{x}}_s, \theta) \mu(d \theta, y) d y \, d s.
\end{split}
\end{equation}
We now use hypothesis \eqref{H} and add the following term in the previous equation
\begin{equation}
0 = \int_J \left[\int_{[\bar{y}]} \left[W(x,y) - W(\bar{x}, y)\right] \int_\bbR \Gamma(\theta^{\bar{x}}_s, \theta) \mu_s(d \theta, \bar{y}) d y \right]d \bar{y},
\end{equation}
where the integrals are sums whenever $J$ or $[\bar{y}]$ are countable (recall that $[\bar{y}]$ is a measurable subset of $[0,1]$ from the hypothesis).

By taking the squares and using $(a+b)^2 \leq 2(a^2 + b^2)$, this leads to
\begin{equation}
\begin{split}
&\left\vert\theta^x_t - \theta^{\bar{x}} _t\right\vert^2 \leq 2T^2 \int_0^t \left\vert \theta^x_s - \theta^{\bar{x}} _s \right\vert^2 d s + \\
&+ 2T^2 \int_0^t \int_J \left\vert \int_{[\bar{y}]} [W(x,y) - W(\bar{x}, y)] \int_\bbR \Gamma(\theta^{\bar{x}}_s, \theta) [\mu_s(d \theta, y) - \mu_s (d \theta, \bar{y})] d y \right\vert^2 d \bar{y} \, d s,
\end{split}
\end{equation}
where we have used Cauchy-Schwartz inequality as well as the fact that $\Gamma$ is 1-Lipschitz. Applying Cauchy-Schwartz again, and using \eqref{d:lipsc}, we are left with
\begin{equation}
\begin{split}
\left\vert\theta^x_t - \theta^{\bar{x}} _t\right\vert^2 \leq 2T^2 \int_0^t \left\vert \theta^x_s - \theta^{\bar{x}} _s \right\vert^2 d s + 4T^2 \int_0^t \int_J \left[\int_{[\bar{y}]} D^2_s (\mu(y), \mu(\bar{y}))d y\right] \, d \bar{y} \, d s,
\end{split}
\end{equation}
Thus, taking the supremum with respect to the time and the expected value $\bE$ this leads to
\begin{equation}
\label{last}
\begin{split}
\bE\left[ \sup_{t \in [0,T]} \left\vert\theta^x_t - \theta^{\bar{x}} _t\right\vert^2 \right] \leq& 2T^2 \int_0^T \bE\left[ \sup_{u \in [0,s]} \left\vert\theta^x_u - \theta^{\bar{x}} _u\right\vert^2 \right] d s + \\
& + 4T^2 \int_0^T \int_J \left[\int_{[\bar{y}]} D^2_s (\mu(y), \mu(\bar{y}))d y\right] \, d \bar{y} \, d s.
\end{split}
\end{equation}
Finally, we can take the maximum with respect to $\bar{x} \in J$ and $x \in [\bar{x}]$ and use the characterization of \eqref{d:wass}, to obtain
\begin{equation}
\maxtwo{\bar{x} \in J}{x \in [\bar{x}]} D^2_T (\mu(\bar{x}), \mu(x)) \leq 6T^2 \int_0^T \maxtwo{\bar{x} \in J}{x \in [\bar{x}]} D^2_s (\mu(\bar{x}), \mu(x)) d s,
\end{equation}
which implies \eqref{goal}.

Whenever $F$ is not zero, the proof is basically the same: using the Lipschitz properties of $F$,  a term equal to $\left\vert \theta^x_s - \theta^{\bar{x}} _s \right\vert^2$ appears, thus adding the constant factor of $2T^2$ to the final equation \eqref{last}.

If $\sigma$ is not constant, we can handle the extra term by using the Burkholder-Davis-Gundy inequality \cite[Theorem 3.28]{karatzasShreve1996}. Indeed, one can bound the expectation of the stochastic term by its quadratic variation. Using the fact that $\sigma$ is Lipschitz, this yields
\begin{equation*}
\bE\left[ \sup_{t \in [0,T]} \left\vert\int_0^T \left(\sigma(\theta^x_s) - \sigma(\theta^{\bar{x}}_s) \right) d B_s \right\vert^2 \right] \leq C T^2 	\int_0^T \bE\left[ \sup_{u \in [0,s]} \left\vert\theta^x_u - \theta^{\bar{x}} _u\right\vert^2 \right] d s,
\end{equation*}
with $C$ a universal positive constant. The rest of the proof remains unchanged.
\qed

\section{Acknowledgments}

The author is thankful to Gianmarco Bet, Giambattista Giacomin and Francesca Nardi for discussions while writing this note.

The author would like to thank the anonymous referee for valuable suggestions on the presentation of the results.

\bibliographystyle{abbrv}
\bibliography{biblio}

\end{document}